\documentclass[a4paper,reqno,10pt]{amsart}

\DeclareRobustCommand{\SkipTocEntry}[5]{}

\usepackage{hyperref}
\usepackage{a4wide}

\usepackage[justification=centering]{caption}
\usepackage{subcaption}

\usepackage{amssymb}
\usepackage{amstext}
\usepackage{amsmath}
\usepackage{amscd}
\usepackage{amsthm}
\usepackage{amsfonts}
\usepackage{stmaryrd}

\usepackage{graphicx}
\usepackage{latexsym}
\usepackage{mathrsfs}

\usepackage{enumitem}
\setlist[enumerate,1]{label={\upshape(\arabic*)}}
\setlist[enumerate,2]{label={\upshape(\alph*)}}

\usepackage{tikz}
\usetikzlibrary{cd,arrows,matrix,backgrounds,shapes,positioning,calc,decorations.markings,decorations.pathmorphing,decorations.pathreplacing}
\tikzset{blackv/.style={circle,fill=black,inner sep=3pt,outer sep=3pt},
         whitev/.style={circle,fill=white,draw=black,inner sep=3pt,outer sep=3pt},
         blabel/.style={
           circle, fill=white, draw=black, font=\scriptsize,
           inner sep=0.5pt, outer sep=0pt},
         redv/.style={circle,fill=red,inner sep=3pt,outer sep=3pt},
         block/.style={draw,rectangle split,rectangle split horizontal,rectangle split parts=#1},
         symbol/.style={
           draw=none,
           every to/.append style={
             edge node={node [sloped, allow upside down, auto=false]{$#1$}}}}
}

\usepackage{rotating}

\usepackage{array}
\newcolumntype{C}{>{$}c<{$}}

\usepackage{makecell}

\newtheorem{theorem}{Theorem}[section]
\newtheorem{theoremi}{Theorem}
\newtheorem{corollaryi}[theoremi]{Corollary}

\newtheorem{corollary}[theorem]{Corollary}
\newtheorem{lemma}[theorem]{Lemma}
\newtheorem*{lemma*}{Lemma}
\newtheorem*{theorem*}{Theorem}
\newtheorem{proposition}[theorem]{Proposition}
\newtheorem{definition-proposition}[theorem]{Definition-Proposition}
\newtheorem{problem}[theorem]{Problem}
\newtheorem{question}[theorem]{Question}

\theoremstyle{definition}
\newtheorem{definition}[theorem]{Definition}

\newtheorem{remark}[theorem]{Remark}
\newtheorem{example}[theorem]{Example}

\renewcommand{\AA}{\mathcal{A}}

\newcommand{\CC}{\mathcal{C}}

\newcommand{\FF}{\mathcal{F}}
\newcommand{\FFF}{\mathsf{F}}

\newcommand{\N}{\mathbb{N}}

\newcommand{\TT}{\mathcal{T}}
\newcommand{\TTT}{\mathsf{T}}

\newcommand{\Ext}{\operatorname{Ext}\nolimits}
\newcommand{\Hom}{\operatorname{Hom}\nolimits}

\newcommand{\Spec}{\operatorname{Spec}\nolimits}

\newcommand{\Image}{\operatorname{Im}\nolimits}
\newcommand{\Kernel}{\operatorname{Ker}\nolimits}
\newcommand{\Cokernel}{\operatorname{Coker}\nolimits}
\newcommand{\Supp}{\operatorname{Supp}\nolimits}
\newcommand{\Ass}{\operatorname{Ass}\nolimits}

\newcommand{\coker}{\Cokernel}
\newcommand{\im}{\Image}
\renewcommand{\ker}{\Kernel}

\DeclareMathOperator{\moduleCategory}{\mathsf{mod}} \renewcommand{\mod}{\moduleCategory}

\DeclareMathOperator{\coh}{\mathsf{coh}}

\DeclareMathOperator{\ice}{\mathsf{ice}}
\DeclareMathOperator{\ike}{\mathsf{ike}}
\DeclareMathOperator{\ie}{\mathsf{ie}}
\DeclareMathOperator{\serre}{\mathsf{serre}}

\DeclareMathOperator{\fl}{\mathsf{fl}}

\DeclareMathOperator{\wide}{\mathsf{wide}}

\DeclareMathOperator{\tors}{\mathsf{tors}}
\DeclareMathOperator{\torf}{\mathsf{torf}}

\DeclareMathOperator{\Sub}{\mathsf{Sub}}
\DeclareMathOperator{\Fac}{\mathsf{Fac}}
\DeclareMathOperator{\Filt}{\mathsf{Filt}}

\newcommand{\iso}{\cong}

\newcommand{\defl}{\twoheadrightarrow}

\newenvironment{sbmatrix}{\left[\begin{smallmatrix}}{\end{smallmatrix}\right]}

\numberwithin{equation}{section}

\begin{document}
\title{IE-closed subcategories of commutative rings are torsion-free classes}

\author[H. Enomoto]{Haruhisa Enomoto}
\address{Graduate School of Science, Osaka Metropolitan University, 1-1 Gakuen-cho, Naka-ku, Sakai, Osaka 599-8531, Japan}
\email{henomoto@omu.ac.jp}

\subjclass[2020]{13C60, 18E10, 18E40}
\keywords{IKE-closed subcategories, IE-closed subcategories, torsion-free classes, commutative noetherian rings}
\begin{abstract}
  Let $\CC$ be a subcategory of the category of finitely generated $R$-modules over a commutative noetheian ring $R$. We prove that, if $\CC$ is closed under images and extensions (which we call an IE-closed subcategory), then $\CC$ is closed under submodules, and hence is a torsion-free class.
  This result complements Stanley--Wang's result in some sense and, furthermore, provides a complete answer to the question posed by Iima--Matsui--Shimada--Takahashi.
  The proof relies on the general theory of IE-closed subcategories in an abelian category, which states that IE-closed subcategories are precisely the intersections of torsion classes and torsion-free classes.
  Additionally, we completely characterize right noetherian rings such that every IE-closed subcategory (or torsion-free class) is a Serre subcategory.
\end{abstract}

\maketitle

\section{Introduction}
We begin by presenting the main result of this paper:
\begin{theoremi}[= Theorem \ref{thm:main}]\label{thm:A}
  Let $R$ be a commutative noetherian ring, and let $\CC$ be a subcategory of the category $\mod R$ of finitely generated $R$-modules. Then the following conditions are equivalent.
  \begin{enumerate}
    \item $\CC$ is closed under submodules and extensions (i.e., $\CC$ is a torsion-free class).
    \item $\CC$ is closed under taking images (i.e., for every $f \colon C_1 \to C_2$ with $C_1, C_2 \in \CC$, we have $\im f \in \CC$) and extensions.
  \end{enumerate}
\end{theoremi}
It is clear that condition (1) implies condition (2) since $\im f$ is a submodule of $C_2$. Remarkably, our main result establishes that the converse holds as well. We now further discuss the context and motivation behind this theorem.

In the study of abelian categories, understanding the structure of subcategories has been a central theme.
One of the earliest and most influential results in this area is the classification of \emph{Serre subcategories} by Gabriel \cite{gab}. He showed that, for a Noetherian scheme $X$, Serre subcategories of the category of coherent sheaves $\coh X$ can be classified in terms of specialization closed subsets of $X$. This result laid the groundwork for subsequent research on subcategories in various abelian categories.

Another important development was the study of \emph{torsion classes} and \emph{torsion-free classes} of an abelian category. Happel--Reiten--Smal\o\ \cite{HRS} proved that they correspond bijectively to certain $t$-structures of the bounded derived category of it. This correspondence has played a crucial role in the study of Bridgeland stability conditions via the simple tilt procedure.
More recently, the representation theory of algebras has seen a surge of interest in the study of torsion(-free) classes by $\tau$-tilting theory of Adachi--Iyama--Reiten \cite{AIR}, which enables us to combinatorial description of subcategories.

In recent years, researchers have also considered \emph{wide subcategories} of abelian categories. For example, Br\"uning \cite{br} related wide subcategories to thick subcategories of derived categories, and Marks--\v{S}t\!'ov\'{i}\v{c}ek \cite{MS} wide subcategories are closely related to torsion(-free) classes with a bijection between the two under certain finiteness conditions.

The author recently proposed the study of more general subcategories in an abelian category, namely, \emph{ICE-closed}, \emph{IKE-closed}, and \emph{IE-closed} subcategories in \cite{eno-mono, eno-rigid-ice, ES1, ES2}. These generalize the aforementioned subcategories, and defined by the condition that they are closed under taking some of the following operations: Images, Cokernels, Kernels, and Extensions (the initials correspond to these classes of subcategories).

For clarity, let us summarize the definitions of the aforementioned subcategories of an abelian category. Figure \ref{fig:rel} summarizes the implications between these subcategories.
\begin{definition}
  Let $\CC$ be a subcategory of an abelian category $\AA$.
  \begin{itemize}
    \item $\CC$ is \emph{Serre} if it is closed under subobjects, quotients, and extensions.
    \item $\CC$ is a \emph{torsion class} if it is closed under quotients and extensions.
    \item $\CC$ is a \emph{torsion-free class} if it is closed under subobjects and extensions.
    \item $\CC$ is \emph{wide} if it is closed under kernels, cokernels, and extensions.
    \item $\CC$ is \emph{ICE-closed} if it is closed under images, cokernels, and extensions.
    \item $\CC$ is \emph{IKE-closed} if it is closed under images, kernels, and extensions.
    \item $\CC$ is \emph{IE-closed} if it is closed under images and extensions.
  \end{itemize}
  Over a right noetherian ring $\Lambda$, we denote by $\serre \Lambda$, $\tors\Lambda$, $\torf\Lambda$, $\wide\Lambda$, $\ice\Lambda$, $\ike\Lambda$, and $\ie\Lambda$, the set of Serre subcategories, torsion classes, torsion-free classes, wide subcategories, ICE-closed subcategories, IKE-closed subcategories, and IE-closed subcategories of the category $\mod\Lambda$ of finitely generated right $\Lambda$-modules respectively.
\end{definition}

\begin{figure}
  \begin{tikzpicture}[
      arrow/.style={-Implies,double equal sign distance,shorten >=1pt,shorten <=1pt},
      xscale=3, yscale=1.2
    ]

    \fill[gray!20, rounded corners]
    (-0.3, 0.5) -- (-0.3, -2.5) -- (0.7, -3.5) --
    (1.5, -3.5) -- (1.5, -1) -- (0.5, 0.5) -- cycle;

    \fill[gray!50, rounded corners]
    (-1.5, -0.5) -- (-1.5, -3.5) -- (-0.5, -4.5) --
    (0.3, -4.5) -- (0.3, -3.5) -- (-0.5, -2.5) -- (-0.5, -0.5) -- cycle;

    \node (Serre) at (0, 0) {Serre};
    \node (torsion-free) at (-1, -1) {torsion-free class};
    \node (torsion) at (1, -1) {torsion class};
    \node (wide) at (0, -2) {wide};
    \node (IKE) at (-1, -3) {IKE-closed};
    \node (ICE) at (1, -3) {ICE-closed};
    \node (IE) at (0, -4) {IE-closed};

    \draw[arrow] (Serre) -- (torsion);
    \draw[arrow] (Serre) -- (torsion-free);
    \draw[arrow] (Serre) -- (wide);
    \draw[arrow] (torsion-free) -- (IKE);
    \draw[arrow] (torsion) -- (ICE);
    \draw[arrow] (wide) -- (IKE);
    \draw[arrow] (wide) -- (ICE);
    \draw[arrow] (ICE) -- (IE);
    \draw[arrow] (IKE) -- (IE);

  \end{tikzpicture}
  \caption{Relations of classes of subcategories of an abelian category}
  \label{fig:rel}
\end{figure}
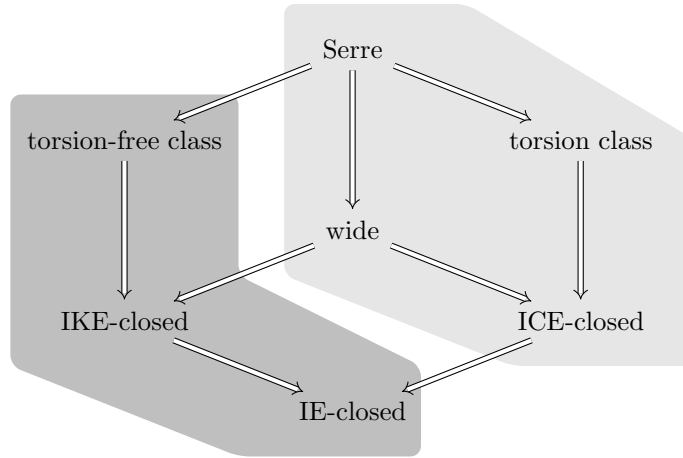

As mentioned earlier, there are scattered studies of these types of subcategories. The natural (and challenging) problem is to consider the following:

\begin{problem}
For a given noetherian ring $\Lambda$, classify and establish relationships between the aforementioned classes of subcategories of $\mod\Lambda$.
\end{problem}

Various works in this direction include those mentioned previously. Since considering general noetherian rings can be difficult, there are two contrasting directions that restrict the class of rings to the following:

\begin{enumerate}
  \item (Non-commutative) finite-dimensional $k$-algebras $\Lambda$ over a field $k$.
  \item Commutative noetherian rings $R$.
\end{enumerate}

In the commutative case, the following result by Stanley--Wang is quite intriguing:
\begin{theorem}[{\cite[Corollary 7.1]{SW}}]\label{thm:SW}
  Let $R$ be a commutative noetherian ring. Then $\serre R = \tors R = \wide R = \ice R$ hold.
\end{theorem}
This result suggests that, in the commutative case, some classes of subcategories may coincide in general, which is a fascinating phenomenon unique to commutative rings.
In this direction, Iima--Matsui--Shimada--Takahashi recently proposed the following question:
\begin{question}[{\cite[Question 3.3]{IMST}}]\label{q:ike-is-torf}
  Let $R$ be a commutative noetherian ring. Is every IKE-closed subcategory of $\mod R$ torsion-free?
\end{question}

They addressed this question for the case of numerical semigroup rings by using the detailed observation of the module category based on careful calculations, and provided a combinatorial sufficient condition for an affirmative answer \cite[Theorem 1.3]{IMST}, together with lots of examples. Moreover, they completely classified IKE-closed subcategories of $\mod R$ for this case.

Motivated by their question and results, this paper investigates torsion-free classes, IKE-closed subcategories, and IE-closed subcategories of $\mod R$. Theorem \ref{thm:A} provides a complete answer to Question \ref{q:ike-is-torf} as follows and, in fact, a slightly stronger result in that we also consider IE-closed subcategories.
\begin{corollaryi}
  Let $R$ be a commutative noetherian ring. Then $\torf R = \ike R = \ie R$ holds.
  In particular, Question \ref{q:ike-is-torf} is affirmative.
\end{corollaryi}
This situation, combined with Theorem \ref{thm:SW}, indicates that the classes of subcategories in each of the two gray regions in Figure \ref{fig:rel} coincide, so they are actually only two classes of subcategories for the commutative case.

Thanks to this theorem, the study of I(K)E-closed subcategories is completely the same as the study of torsion-free classes for the commutative case. For example, we can apply Takahashi's classification \cite{takahashi} of torsion-free classes using $\Spec R$.
Summarizing Takahashi's, and Gabriel's and our results, we obtain the following commutative diagram with the horizontal maps being bijective:
\[
  \begin{tikzcd}
    \serre R = \tors R = \wide R = \ice R \rar["\Supp"] \dar[symbol=\subseteq]
    & \{ \text{specialization-closed subsets of $\Spec R$} \} \dar[symbol=\subseteq] \\
    \torf R = \ike R = \ie R \rar["\Ass"] & \{ \text{subsets of $\Spec R$} \}
  \end{tikzcd}
\]
Here, for a subcategory $\CC$ of $\mod R$, we write $\Supp \CC := \bigcup \{ \Supp_R X \mid X \in \CC \}$ and $\Ass \CC := \bigcup \{ \Ass_R X \mid X \in \CC \}$.

One more natural question is when the all of these classes of subcategories are the same. Actually, this characterizes artinian rings as follows.
\begin{corollaryi}[= Corollary \ref{cor:art-char}]
  Let $R$ be a commutative noetherian ring $R$. Then the following conditions are equivalent:
  \begin{enumerate}
    \item $R$ is artinian.
    \item $\serre R = \tors R = \torf R = \wide R = \ice R = \ike R = \ie R$ holds.
  \end{enumerate}
\end{corollaryi}

Although this fact can be easily deduced from the classifications of subcategories due to Gabriel and Takahashi, we derive it from the following complete characterization of the same situation for a non-commutative noetherian ring.
\begin{theoremi}[= Theorem \ref{thm:char-all-same}]
  Let $\Lambda$ be a right noetherian ring. Then the following conditions are equivalent:
  \begin{enumerate}
    \item $\Lambda$ is Morita equivalent to a finite direct product of right artinian local rings.
    \item $\torf\Lambda = \serre \Lambda$ holds.
    \item  $\serre \Lambda = \tors \Lambda = \torf \Lambda = \wide \Lambda = \ice \Lambda = \ike \Lambda = \ie \Lambda$ holds.
  \end{enumerate}
\end{theoremi}

Most of the aforementioned results are based on the following general statement regarding IE-closed subcategories, which is valid for any abelian category and is of interest on its own.
\begin{theoremi}[= Corollary \ref{cor:ie-as-inter}]
  Let $\AA$ be an abelian category and $\CC$ a subcategory of $\AA$. Then the following conditions are equivalent.
  \begin{enumerate}
    \item $\CC$ is IE-closed.
    \item There exist a torsion class $\TT$ and a torsion-free class $\FF$ such that $\CC = \TT \cap \FF$.
  \end{enumerate}
\end{theoremi}
This result generalizes \cite[Lemma 4.23]{eno-from-tors}, where $\AA$ is assumed to be an abelian length category.

\addtocontents{toc}{\SkipTocEntry}
\subsection*{Conventions and notation}
By a \emph{subcategory} of an additive category $\CC$, we always mean a full additive subcategory that is closed under isomorphisms.
For a ring $\Lambda$, we denote the category of finitely generated right $\Lambda$-modules by $\mod\Lambda$ (which is abelian if $\Lambda$ is right noetherian).

\section{IE-closed subcategories via torsion/torsion-free classes}
Throughout this section, we always denote by $\AA$ an abelian category, without assuming any restrictions such as being an abelian length category.
In this section, we focus on IE-closed subcategories of general abelian categories and their relationship with torsion and torsion-free classes. To begin, let us recall the definitions of various subcategories of an abelian category.

To begin, let us recall several subcategories being ``closed under specific operations" within an abelian category $\AA$.
\begin{definition}\label{def:basic-def}
  Let $\CC$ be a subcategory of $\AA$.
  \begin{enumerate}
    \item $\CC$ is \emph{closed under extensions} if, for every short exact sequence
          \[
            \begin{tikzcd}
              0 \rar & L \rar & M \rar & N \rar & 0,
            \end{tikzcd}
          \]
          in $\AA$ with $L,N \in \CC$, we have $M \in \CC$.
    \item $\CC$ is \emph{closed under quotients (resp. subobjects)} if, for every object $C \in \CC$, every quotient (resp. submodule) of $C$ in $\mod\Lambda$ belongs to $\CC$.
    \item $\CC$ is closed under \emph{images (resp. kernels, cokernels)} if, for every map $\varphi \colon C_1 \to C_2$ with $C_1, C_2 \in \CC$, we have $\im\varphi \in \CC$ (resp. $\ker\varphi\in\CC$, $\coker\varphi\in\CC$).
  \end{enumerate}
\end{definition}

Next, we recall the classes of subcategories discussed in the introduction:

\begin{definition}
  Let $\CC$ be a subcategory of $\AA$.
  \begin{enumerate}
    \item $\CC$ is a \emph{torsion class} (resp. \emph{torsion-free class}) if $\CC$ is closed under quotients and extensions (resp. subobjects and extensions).
    \item $\CC$ is \emph{wide} if $\CC$ is closed under kernels, cokernels, and extensions.
    \item $\CC$ is \emph{ICE-closed} (resp. \emph{IKE-closed}) if $\CC$ is closed under images, cokernels, and extensions (resp. images, kernels, and extensions).
    \item $\CC$ is \emph{IE-closed} if $\CC$ is closed under images and extensions.
  \end{enumerate}
\end{definition}

\subsection{Torsion(-free) classes and torsion pairs}

It is worth mentioning the relationship between torsion(-free) classes and \emph{torsion pairs} in $\AA$. A pair $(\TT, \FF)$ of $\AA$ is called a \emph{torsion pair} if the following conditions are satisfied:
\begin{enumerate}
  \item $\AA(\TT, \FF) = 0$.
  \item For any $X \in \AA$, there exists a short exact sequence
        \[
          \begin{tikzcd}
            0 \rar & T \rar & X \rar & F \rar & 0
          \end{tikzcd}
        \]
        in $\AA$ with $T \in \TT$ and $F \in \FF$.
\end{enumerate}

It is easy to see that if $(\TT, \FF)$ is a torsion pair in $\AA$, then $\TT$ is a torsion class and $\FF$ is a torsion-free class. We refer to a \emph{strong torsion(-free) class} as $\TT$ and $\FF$ when they form a torsion pair $(\TT, \FF)$ in $\AA$. Note that some authors use the terminology \emph{torsion(-free) classes} for strong ones. In some situations, strong and usual torsion(-free) classes are equivalent or related as follows:

\begin{itemize}
  \item If $\AA$ is noetherian (meaning that every object is noetherian), then strong torsion classes coincide with torsion classes.
  \item Dually, if $\AA$ is artinian, then strong torsion-free classes coincide with torsion-free classes.
  \item As for ``large" categories, suppose that $\AA$ is well-powered and cocomplete. Then a subcategory $\TT$ of $\AA$ is a strong torsion class if and only if $\TT$ is a torsion class and closed under coproducts. We can also dualize this statement for (strong) torsion-free classes.
\end{itemize}

The proofs of these facts are well-known and standard, so we omit them. Furthermore, we note the following proposition which characterizes artinian rings in terms of torsion classes and torsion pairs:
\begin{proposition}
  Let $\Lambda$ be a right noetherian ring. Then the following conditions are equivalent:
  \begin{enumerate}
    \item $\Lambda$ is a right artinian ring.
    \item Every torsion-free class $\FF$ in $\mod\Lambda$ is a strong torsion class.
  \end{enumerate}
\end{proposition}
\begin{proof}
  (1) $\Rightarrow$ (2): This direction follows immediately from the fact that $\mod \Lambda$ is an artinian abelian category.

  (2) $\Rightarrow$ (1):
  Let $\fl\Lambda$ be the subcategory of $\mod\Lambda$ consisting of finitely generated $\Lambda$-modules with finite length. Since $\fl\Lambda$ is a Serre subcategory of $\mod\Lambda$, it is also a torsion-free class in $\mod\Lambda$. Hence, there exists a torsion class $\TT$ in $\mod\Lambda$ such that $(\TT, \fl \Lambda)$ is a torsion pair in $\mod\Lambda$ by (2).

  We claim that $\TT = 0$. Suppose for a contradiction that $\TT$ contains a non-zero module $M$. Since $M$ is finitely generated, it has at least one maximal submodule $N$. There exists a surjective homomorphism $\varphi: M \to M/N$ to a simple module $M/N$. However, since $M/N$ belongs to $\fl\Lambda$, we have $\Hom_\Lambda(M, M/N) = 0$ by the torsion pair property. This contradicts the existence of the non-zero surjective homomorphism $\varphi$, so we must have $M=0$ and thus $\TT=0$.

  Therefore, we have shown that $(0, \fl\Lambda)$ is a torsion pair in $\mod \Lambda$, which implies that $\fl\Lambda = \mod\Lambda$. Hence, $\Lambda_\Lambda$ has finite length, and so $\Lambda$ is right artinian.
\end{proof}

\subsection{Description of torsion(-free) closures and IE-closed subcategories}
We will now introduce the concept of torsion(-free) closure and provide an explicit constructive description for them.
Let $\CC$ be a collection of objects in an abelian category $\AA$. We denote by $\TTT(\CC)$ (resp. $\FFF(\CC)$) the smallest torsion class containing $\CC$ (resp. the smallest torsion-free class containing $\CC$).
To see the existence of $\TTT(\CC)$, we can consider the intersection of all torsion classes in $\AA$ that contain $\CC$. It is easy to see that this intersection is again a torsion class containing $\CC$ and hence is the smallest torsion class containing $\CC$.

However, for our purposes, it is important to have a more explicit and constructive description of $\TTT(\CC)$. To achieve this, we introduce the following notation.
\begin{definition}
  Let $\CC$ be a collection of objects in $\AA$.
  \begin{enumerate}
    \item $\Fac \CC$ denotes the subcategory of $\AA$ consisting of $X$ such that there exists a surjection $C \defl X$ in $\AA$ from some object $C \in \CC$.
    \item $\Sub \CC$ denotes the subcategory of $\AA$ consisting of $X$ such that there exists an injection $X \hookrightarrow C$ in $\AA$ to some object $C \in \CC$.
    \item $\Filt \CC$ denotes the subcategory of $\AA$ consisting of $X$ such that there is a finite sequence of subobjects of $X$:
          \[
            0 = X_0 \subseteq X_1 \subseteq X_2 \cdots \subseteq X_n = X,
          \]
          such that $X_i/X_{i-1} \in \CC$ for each $i$.
          In this case, we say that \emph{$X$ has a $\CC$-filtration of length $n$}.
  \end{enumerate}
\end{definition}
We can provide a more constructive description of $\TTT(\CC)$ as follows. While the case of $\mod\Lambda$ for a finite-dimensional $k$-algebra over a field $k$ has been proven in \cite[Lemma 3.1]{MS}, the same proof applies to our setting. We include the proof below for the convenience of the reader, as it is not lengthy.
\begin{lemma}\label{lem:tors-closure}
  Let $\CC$ be a collection of objects in $\AA$. Then we have the equality:
  \[
    \TTT(\CC) = \Filt (\Fac \CC).
  \]
\end{lemma}
\begin{proof}
  By construction, if a torsion class $\TT$ contains $\CC$, then it must also contain $\Fac \CC$, since $\TT$ is closed under quotients, and then $\Filt(\Fac \CC)$, since $\TT$ is closed under extensions. Therefore, it is sufficient to show that $\Filt(\Fac \CC)$ is a torsion class.

  It is evident that $\Filt(\Fac \CC)$ is closed under extensions. Hence, we need to establish that it is also closed under quotients. We will prove the following claim by induction on $n$: if $X\in\Filt(\Fac\CC)$ has a $(\Fac\CC)$-filtration length $n$, then every quotient of $X$ belongs to $\Filt(\Fac\CC)$.

  Suppose that $X\in\Filt(\Fac\CC)$ has a $(\Fac\CC)$-filtration length $n$. If $n=0$, then $X=0$, and every quotient of $X$ trivially belongs to $\Filt(\Fac\CC)$. Next, suppose that $n>0$. Then there exists a short exact sequence
  \[
    \begin{tikzcd}
      0 \rar & X' \rar["\iota"] & X \rar["\pi"] & X'' \rar & 0
    \end{tikzcd}
  \]
  in $\AA$, such that $X'\in\Fac\CC$ and $X''\in\Filt(\Fac\CC)$ has a $(\Fac\CC)$-filtration of length $n-1$. Now, consider any quotient $p\colon X\twoheadrightarrow Y$. We obtain the following exact commutative diagram by considering the image of $p\iota\colon X'\to Y$:
  \[
    \begin{tikzcd}
      0 \rar & X' \rar["\iota"] \dar[twoheadrightarrow] & X
      \rar["\pi"] \dar[twoheadrightarrow, "p"] & X'' \rar \dar["f"] & 0 \\
      0 \rar & \im (p \iota) \rar["\iota'"'] & Y \rar["\pi'"'] & B \rar & 0
    \end{tikzcd}
  \]
  Here, $B$ is the cokernel of $\iota'$. Thus, $f$ is a surjection since $\pi'p$ is a surjection. Additionally, $\im (p \iota)\in\Fac\CC$ as $X'\in\Fac\CC$ and $\Fac\CC$ is closed under quotients. By the induction hypothesis, $B$ belongs to $\Filt(\Fac\CC)$ since $B$ is a quotient of $X''$. Therefore, both $\im (p \iota)$ and $B$ belong to $\Filt(\Fac\CC)$, and we conclude that $Y\in\Filt(\Fac\CC)$. This completes the proof.
\end{proof}

Since we are working with general abelian categories, we can dualize the previous lemma to obtain the following description of the torsion-free closure.
\begin{lemma}\label{lem:torf-closure}
  Let $\CC$ be a collection of objects in $\AA$. Then we have the equality:
  \[
    \FFF(\CC) = \Filt (\Sub \CC).
  \]
\end{lemma}

We can now prove a key observation, which is of central importance to our work.
\begin{theorem}\label{thm:ie-as-inter-of-closure}
  Let $\AA$ be an abelian category and let $\CC$ be an IE-closed subcategory of $\AA$. Then we have the equality $\CC = \TTT(\CC) \cap \FFF(\CC)$.
\end{theorem}
\begin{proof}
  The proof of this claim has already been established in \cite[Lemma 4.23]{eno-from-tors} under the assumption that $\AA$ is an abelian length category. However, since the proof therein only depends on the descriptions of $\TTT(\CC)$ and $\FFF(\CC)$ given in Lemmas \ref{lem:tors-closure} and \ref{lem:torf-closure}, respectively, the same proof applies in this general setting. Therefore, we avoid repeating the proof here.
\end{proof}
Since all torsion classes and torsion-free classes are IE-closed, and IE-closed subcategories are closed under intersections, using Theorem \ref{thm:ie-as-inter-of-closure}, we immediately obtain the following equivalent condition for a subcategory to be IE-closed. This extends \cite[Proposition 2.3]{ES2}, where only the case $\AA = \mod\Lambda$ for an artin algebra $\Lambda$ is considered.

\begin{corollary}\label{cor:ie-as-inter}
  Let $\CC$ be a subcategory of $\AA$. The following conditions are equivalent:
  \begin{enumerate}
    \item $\CC$ is an IE-closed subcategory of $\AA$.
    \item There exist a torsion class $\TT$ and a torsion-free class $\FF$ such that $\CC = \TT \cap \FF$.
  \end{enumerate}
\end{corollary}

\section{Applications}
We are now ready to apply our theory to the study of commutative noetherian rings. In what follows, we will denote by $R$ a commutative noetherian ring.

Firstly, we can prove the main result of this paper, which completely characterizes I(K)E-closed subcategories of $\mod R$ and provides an affirmative answer to Question \ref{q:ike-is-torf}.
\begin{theorem}\label{thm:main}
  Let $\CC$ be a subcategory of $\mod R$. Then the following conditions are equivalent:
  \begin{enumerate}
    \item $\CC$ is a torsion-free class.
    \item $\CC$ is an IKE-closed subcategory.
    \item $\CC$ is an IE-closed subcategory.
  \end{enumerate}
\end{theorem}

\begin{proof}
  The implications (1) $\Rightarrow$ (2) $\Rightarrow$ (3) are clear, so we only need to prove (3) $\Rightarrow$ (1).

  Let $\CC$ be an IE-closed subcategory of $\mod R$. By Corollary \ref{cor:ie-as-inter}, there exist a torsion class $\TT$ and a torsion-free class $\FF$ of $\mod R$ such that $\CC = \TT \cap \FF$.
  It follows from Stanley--Wang's result (Theorem \ref{thm:SW}) that $\TT$ is a Serre subcategory. In particular, $\TT$ is a torsion-free class since every Serre subcategory of $\AA$ is a torsion-free class. Therefore, $\CC$ is an intersection of two torsion-free classes, namely $\TT$ and $\FF$. This implies that $\CC$ is also a torsion-free class. Hence, the proof is complete.
\end{proof}

\begin{remark}
  We can generalize Theorem \ref{thm:main} to the non-commutative setting as follows. Let $R$ be a commutative noetherian ring, and let $\Lambda$ be a module-finite $R$-algebra. Suppose that $\Lambda_\mathfrak{p}$ is Morita equivalent to a local ring for every $\mathfrak{p} \in \Spec R$. Then every IE-closed subcategory of $\mod \Lambda$ is a torsion-free class. This can be proved by replacing Stanley-Wang's result in the proof with Iyama--Kimura's corresponding result \cite[Proposition 3.19]{IK}, which states that every torsion class is a Serre subcategory in this setting.
\end{remark}

As a consequence of Takahashi's work \cite{takahashi}, which classifies torsion-free classes of $\mod R$ in terms of subsets of $\Spec R$, we obtain the following immediate corollary of our main result.

\begin{corollary}
  The following statements hold.
  \begin{enumerate}
    \item There exists a bijection between $\torf R = \ike R = \ie R$ and the set of subsets of $\Spec R$.
    \item Let $R$ be a local domain with $\dim R = 1$. Then there are precisely four elements in $\torf R = \ike R = \ie R$: the zero subcategory, $\fl R$, the category of torsion-free $R$-modules, and $\mod R$.
  \end{enumerate}
\end{corollary}

\begin{proof}
  (1) This is an immediate consequence of Theorem \ref{thm:main} and \cite[Theorem B]{takahashi}.

  (2) This also easily follows from Takahashi's result. We refer the reader to \cite[Remark 4.9]{IMST} for details.
\end{proof}

This result extends \cite[Theorem 1.3 (2)]{IMST}, where the above result was established only for a completed numerical semigroup ring $R = k\llbracket H \rrbracket$ associated with a numerical semigroup $H\leq \N$ satisfying certain conditions. Our result applies to $R = k\llbracket H\rrbracket$ for every numerical semigroup $H\leq \N$ such that $\N\setminus H$ is a finite set.

As stated in the introduction and illustrated in Figure \ref{fig:rel}, our main result Theorem \ref{thm:main} and Stanley-Wang's result Theorem \ref{thm:SW} show that there are two distinct classes of subcategories of $\mod R$: (i) $\serre R = \tors R = \wide R = \ice R$, and (ii) $\torf R = \ike R = \ie R$. Moreover, the former class is contained in the latter.

In the rest of this paper, we address the natural question of when these two classes coincide. As our observation applies to non-commutative rings, we show the following general result, which provides a complete characterization of rings where all of the considered classes coincide.
\begin{theorem}\label{thm:char-all-same}
  Let $\Lambda$ be a right noetherian ring. Then the following conditions are equivalent:
  \begin{enumerate}
    \item $\Lambda$ is Morita equivalent to a finite direct product of right artinian local rings.
    \item $\serre \Lambda = \ie \Lambda$ holds (or equivalently, $\serre \Lambda = \tors \Lambda = \torf \Lambda = \wide \Lambda = \ice \Lambda = \ike \Lambda = \ie \Lambda$ holds).
    \item $\serre \Lambda = \torf \Lambda$ holds.
  \end{enumerate}
\end{theorem}
\begin{proof}
  (1) $\Rightarrow$ (2):
  We may assume that $\Lambda$ is a finite direct product of right artinian local rings. Without loss of generality, we may further assume that $\Lambda$ is itself right artinian local, since considering IE-closed (resp. Serre) subcategories of $\mod\Lambda$ is equivalent to considering those for each ring-indecomposable summand of $\Lambda$.

  Let $\CC$ be a non-zero IE-closed subcategory of $\mod\Lambda$, and let $S$ denote the unique simple $\Lambda$-module.
  Suppose that there exists a non-zero module $0 \neq X \in \CC$. As $\Lambda$ is right artinian, there is a surjection $X \twoheadrightarrow S$ and an injection $S \hookrightarrow X$. Thus, $S$ is the image of the composition $X \twoheadrightarrow S \hookrightarrow X$, which means that $S$ belongs to $\CC$ since $\CC$ is closed under images.
  As a result, we have $\Filt S \subseteq \CC$ since $\CC$ is closed under extensions. However, since $\Lambda$ is artinian local, we have $\mod\Lambda = \Filt S$. Therefore, we obtain $\CC = \mod\Lambda$. This implies that the only possible IE-closed subcategories of $\mod\Lambda$ are the zero subcategory and $\mod\Lambda$ itself, and both of these are Serre subcategories. Hence, we have $\serre R = \ie R$.

  (2) $\Rightarrow$ (3): This implication is immediate since every torsion-free class is IE-closed.

  (3) $\Rightarrow$ (1):
  We divide the proof into two steps.

  \textbf{Step 1: $\Lambda$ is right artinian}:
  Consider the subcategory $\FF$ of $\mod\Lambda$ consisting of modules $F$ such that for every simple $\Lambda$-module $S$, we have $\Hom_\Lambda(S, F) = 0$. It is clear that $\FF$ is a torsion-free class, and hence $\FF$ is a Serre subcategory by the assumption.
  If there exists a nonzero module $0 \neq F \in \FF$, then since $F$ is finitely generated, there exists a surjection $F \twoheadrightarrow S$ onto some simple module $S$. Then, since $\FF$ is Serre, we must have $S \in \FF$, which is a contradiction since $\Hom_\Lambda(S, S) \neq 0$.
  Therefore, we obtain $\FF = 0$. This means that for every nonzero module $M$, there exists some simple module $S$ and a nonzero homomorphism $S \to M$. Since $S$ is simple, this is injective. Therefore, $M$ has a simple submodule. To summarize, every nonzero module has a simple submodule.

  It is now a standard exercise in ring theory to prove that $\Lambda$ is right artinian from this condition and the assumption that $\Lambda$ is right noetherian. We include the argument here for the convenience of the reader.
  Consider $\Lambda_\Lambda$. If $\Lambda \neq 0$, then it has a simple submodule $S_1$, and we write $M_1 := S_1$. Then consider the quotient $\Lambda/M_1$. If $\Lambda/M_1 \neq 0$, then this again has a simple submodule $S_2$, which corresponds to a submodule $M_2$ of $\Lambda$ such that $M_1 \subsetneq M_2 \subseteq \Lambda$ and $M_2/M_1 \iso S_2$.
  We can iterate this process to obtain the ascending chain of right ideals of $\Lambda$: $0 \subsetneq M_1 \subsetneq M_2 \subsetneq M_3 \subsetneq \cdots$. Since $\Lambda$ is right noetherian, this sequence must terminate, so we cannot iterate this process infinitely. This means that there exists some $l$ such that $\Lambda/M_l = 0$, that is, $M_l = \Lambda$. Then the sequence $0 \subsetneq M_1 \subsetneq M_2 \subsetneq \cdots \subsetneq M_l = \Lambda$ clearly gives a composition series of $\Lambda$, so $\Lambda$ is right artinian.

  \textbf{Step 2: $\Lambda$ is Morita equivalent to a direct product of local rings}:
  Let $S_1, S_2, \dots, S_n$ be the complete set of non-isomorphic simple $\Lambda$-modules, which is finite since $\Lambda$ is right artinian. To show the claim, it suffices to prove $\Ext_\Lambda^1(S_i, S_j) = 0$ for every $i \neq j$. To do this, we will show that every short exact sequence
  \[
    \begin{tikzcd}
      0 \rar & S_j \rar & M \rar["\pi"] & S_i \rar & 0
    \end{tikzcd}
  \]
  in $\mod\Lambda$ splits if $i \neq j$.
  Consider the torsion-free closure $\FFF(M)$ of $M$. By assumption, this torsion-free class is a Serre subcategory, so it follows that $S_i \in \FFF(M)$ by $M \in \FFF(M)$. On the other hand, we have $\FFF(M) = \Filt (\Sub M)$ by Lemma \ref{lem:torf-closure}, so $S_i \in \Filt(\Sub M)$. However, since $S_i$ is simple, we must have $S_i \in \Sub M$, so we obtain an surjection $i \colon S_i \hookrightarrow M$. By the Jordan-H\"older theorem, we must have $\coker i \iso S_j$, which gives us the following diagram:
  \[
    \begin{tikzcd}
      & & 0 \dar \\
      & & S_i \dar["i"'] \\
      0 \rar & S_j \rar & M \dar \rar["\pi"] & S_i \rar & 0 \\
      & & S_j \dar \\
      & & 0
    \end{tikzcd}
  \]
  If $\pi i = 0$, then $i$ factors through $S_j$. However, Schur's lemma tells us that $\Hom_\Lambda(S_i, S_j) = 0$, so we must have $i = 0$, which leads to a contradiction. Therefore, we have $\pi i \neq 0$. Then, again by Schur's lemma, $\pi i$ is an isomorphism. This implies that $\pi$ is a retraction, so the original short exact sequence splits.
\end{proof}

Applying this observation to the commutative setting, we immediately obtain a categorical characterization of commutative artinian rings as follows.
We remark that this can be also proved using Gabriel and Takahashi's classification of Serre subcategories and torsion-free classes of $\mod R$ using $\Spec R$.
\begin{corollary}\label{cor:art-char}
  For a commutative noetherian ring $R$, the following conditions are equivalent:
  \begin{enumerate}
    \item $R$ is artinian.
    \item $\torf R = \serre R$ holds.
    \item $\serre R = \tors R = \torf R = \wide R = \ice R = \ike R = \ie R$ holds.
  \end{enumerate}
\end{corollary}
\begin{proof}
  It is well-known that a commutative artinian ring is a finite direct product of a commutative artinian local ring. Therefore, this follows from Theorem \ref{thm:char-all-same}.
\end{proof}

We conclude this paper by providing examples of how these types of properties (the coincidence of certain classes of subcategories) typically fail for non-local non-commutative rings.
\begin{example}\label{ex:a2}
  Let $k$ be a field, and consider the upper triangular matrix algebra $\Lambda := \begin{sbmatrix}k & k \\ 0 & k\end{sbmatrix}$, which is a 3-dimensional subalgebra of the matrix ring $M_2(k)$. It is well-known that there are exactly three indecomposable $\Lambda$-modules $A, B, C$ with a non-split short exact sequence
  \[
    \begin{tikzcd}
      0 \rar & A \rar & B \rar & C \rar & 0.
    \end{tikzcd}
  \]
  Since the classification of torsion(-free) classes is well-known for this algebra, we can use Theorem \ref{thm:ie-as-inter-of-closure} to enumerate all IE-closed subcategories (or one can use results of \cite{ES2}). The Hasse diagram of the IE-closed subcategories is shown in Figure \ref{fig:hasse}, where we only display the indecomposable modules contained in each subcategory.

  \begin{figure}
    \begin{tikzcd}[sep=tiny]
      & \{A, B, C\} \ar[dl, -] \ar[dr, -] \\
      \{A, B\} \ar[dd, -] \ar[dr, -] & & \{B, C\} \ar[dd, -] \ar[dl, -] \\
      & \{B\} \ar[dd, -] \\
      \{A\} \ar[dr, -] & & \{C \} \ar[dl, -] \\
      & \{ \}
    \end{tikzcd}
    \caption{The Hasse diagram of IE-closed subcategories of $\mod \Lambda$}
    \label{fig:hasse}
  \end{figure}
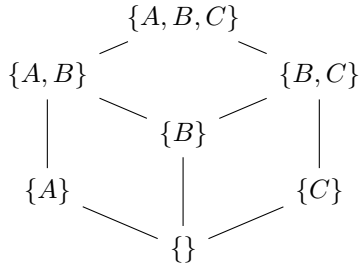

  In Table \ref{tab:list}, we list all the subcategories of each class. From this table, one can observe that \emph{none of these classes coincide}.

  \begin{table}[htp]
    \begin{tabular}{l|l|l}
      Classes    & Lists of subcategories                                       & Numbers \\ \hline \hline
      $\serre R$ & $\{\}, \{A\}, \{C\}, \{A, B, C\}$                            & 4       \\ \hline
      $\tors R$  & $\{\}, \{A\}, \{C\}, \{B, C\}, \{A, B, C\}$                  & 5       \\ \hline
      $\torf R$  & $\{\}, \{A\}, \{C\}, \{A, B\}, \{A, B, C\}$                  & 5       \\ \hline
      $\wide R$  & $\{\}, \{A\}, \{B\}, \{C\}, \{A, B, C\}$                     & 5       \\ \hline
      $\ice R$   & $\{\}, \{A\}, \{B\}, \{C\}, \{B, C\}, \{A, B, C\}$           & 6       \\ \hline
      $\ike R$   & $\{\}, \{A\}, \{B\}, \{C\}, \{A, B\}, \{A, B, C\}$           & 6       \\ \hline
      $\ie R$    & $\{\}, \{A\}, \{B\}, \{C\}, \{A, B\}, \{B, C\}, \{A, B, C\}$ & 7
    \end{tabular}
    \caption{The lists of each class of subcategories of $\mod\Lambda$}
    \label{tab:list}
  \end{table}
\end{example}

\addtocontents{toc}{\SkipTocEntry}
\subsection*{Acknowledgement}
The author would like to extend their gratitude to Osamu Iyama and Shunya Saito for their helpful discussions.
This work is supported by JSPS KAKENHI Grant Number JP21J00299.

\end{document}